\documentclass{article}
\usepackage[utf8]{inputenc}
 \usepackage{amsmath, amsthm, amssymb, enumerate, multicol, chngcntr}
 \usepackage{xcolor}

 \theoremstyle{plain}
 \newtheorem{Thm}{Theorem}[section]
 
 \newtheorem{Cor}[Thm]{Corollary}
 \newtheorem{Lem}[Thm]{Lemma}

 \theoremstyle{definition}
 \newtheorem{Def}[Thm]{Definition}
 
 \newtheorem{quest}[Thm]{Question}

 \counterwithin{equation}{section}

 \newcommand{\FG}{F[G]}
 \newcommand{\FH}{F[H]}
 \newcommand{\FK}{F[K]}
 \newcommand{\FZ}{F[\Z]}
 \renewcommand{\S}{\mathfrak{S}}
 \newcommand{\T}{\mathfrak{T}}
 \newcommand{\D}{\mathcal{D}}
 
 \renewcommand{\H}{\mathcal{H}}
 \newcommand{\Z}{\mathcal{Z}}

 \DeclareMathOperator{\supp}{supp}

 \DeclareMathOperator*{\Aut}{Aut}

 \title{On Schur Rings over Infinite Groups III}
 \author{Nicholas Bastian\footnote{Nicholas Bastian, Southern Utah University, nicbastian16@gmail.com}, Andrew Misseldine\footnote{Corresponding Author: Andrew Misseldine, Southern Utah University, andrewmisseldine@suu.edu, (telephone) 435-865-8228}}

 \date{\today}

\begin{document}

 \maketitle

 \begin{abstract}
 In the paper, we develop further the properties of Schur rings over infinite groups, with particular emphasis on the virtually cyclic group $\Z\times\Z_p$, where $p$ is a prime. We provide structure theorems for primitive sets in these Schur rings. In the case of Fermat and safe primes, a complete classification theorem is proven, which states that all Schur rings over $\Z\times\Z_p$ are traditional. We also draw analogs between Schur rings over $\Z\times\Z_p$ and partitions of difference sets over $\Z_p$.
 \end{abstract}



 \section{Introduction}\label{sec:schur}
 The paper continues developing the theory of Schur rings over infinite groups begun in \cite{InfiniteI, Humphries19}. Schur rings (or S-rings) were first studied by Wielandt \cite{Wielandt49} but only in the case of finite groups. Schur rings are of great interest in algebraic combinatorics because of their connections to association schemes, Cayley graphs, difference sets, and related objects. Let $\Z=\langle z\rangle$ be the infinite cyclic group, written multiplicatively, and, for any positive integer $n$, let $\Z_n=\langle a\rangle$ be the cyclic group of order $n$, also written multiplicatively. In \cite{InfiniteI}, it was shown that there are exactly two Schur rings over $\Z$ and countably many Schur rings over $\Z\times \Z_2$ belonging to one of four types. This paper will extend these results to $\Z\times \Z_n$.

 We remind the reader of some important notation and terminology introduced in \cite{InfiniteI}. Let $\FG$ denote the group algebra over group $G$ and field of coefficients $F$, which we assume has characteristic $0$. Suppose $\alpha=\sum_{g \in G} \alpha_gg
 \in \FG$. Then define $\alpha^{\ast}=\sum_{g\in G} \alpha_g g^{-1}$, which gives an involution on $\FG$, and the \emph{Hadamard product} $\alpha\circ \beta=\sum_{g\in G} \alpha_g\beta_g g$, which gives a commutative, associative, bilinear multiplication on $\FG$. For any finite subset $C\subseteq G$, define $\overline{C} = \sum_{g\in C} g$, called a \emph{simple quantity}. Similarly, let $C^*$ denote the set of inverses of elements of $C$. Note that for subsets $C, D\subseteq G$, we have $\overline{C}\circ \overline{D}=\overline{C\cap D}$ and, if $C$ and $D$ are disjoint, $\overline{C}+\overline{D}=\overline{C\cup D}$. 
 
 Let $\D$ be a partition of the group $G$ of finite support, that is, if $C\in \D$ then $|C| < \infty$. Let $\S=\text{Span}_F\{\overline{C} \mid C\in \D\} \subseteq \FG$. We say that $\S$ is a \emph{Schur ring} if additionally
 \begin{enumerate}[(i)]
 \item $\{1\} \in \D$
 \item if $C \in \D$, then $C^* \in \D$
 \item\label{item:sring} for all $C, D \in \D$,
 $\overline{C}\,\overline{D} = \sum_{E\in \D} \lambda_{CDE}\overline{E}$,
 where all but finitely many $\lambda_{CDE}$ are equal to 0. (The coefficient $\lambda_{CDE}$ is called the \emph{multiplicity} of $\overline{E}$ in the product $\overline{C}\, \overline{D}$.) 
 \end{enumerate} 
Note that Schur rings are subrings of $\FG$ closed under $\circ$ and $*$.

  Leung and Man \cite{LeungII, LeungI} classified all Schur rings over cyclic groups of finite order. Particularly, they showed that all Schur rings over finite cyclic groups are either trivial, automorphic, direct products, or wedge products (see Section \ref{sec:Prelim} for these definitions). We call Schur rings \emph{traditional} if it is one of these four types. Similarly, we say that a group $G$ is \emph{traditional} if all of its Schur rings are traditional. Hence, Leung and Man showed that finite cyclic groups are traditional. In \cite{InfiniteI}, it was shown that $\Z$ and $\Z\times \Z_2$ were traditional. In the case of $\Z$, there are only two Schur rings, both automorphic, namely the discrete and symmetric Schur rings $\FZ$ and $\FZ^\pm$, respectively. For $\Z\times \Z_2$, it was shown that all Schur rings are of one of the following forms:
 \begin{enumerate}[(i)]
 \item $F[\Z_2]\wedge F[\Z]$ or $F[\Z_2]\wedge F[\Z]^\pm$,
 \item $F[H\times \Z_2] \wedge F[\Z]$ or $F[H\times \Z_2]^\pm \wedge F[\Z]^\pm$,
 \item $F[H\times \Z_2]^{\langle \psi\rangle}\wedge F[\Z]^\pm$ or $F[\Z\times\Z_2]^{\langle \psi\rangle}$,
 \item $F[\Z\times\Z_2]$ or $F[\Z\times\Z_2]^\pm$,
 \end{enumerate} where $1<H \le \Z$ and $\psi : \Z\times \Z_2 \to \Z\times \Z_2$ is the automorphism induced by the relation $\psi : z\mapsto az^{-1}$, $a\mapsto a$. 

 In this paper, we prove an analogous result for the group $\Z\times \Z_p$, where $p$ belongs to one of two families of odd primes, namely the Fermat and safe primes. A prime $p$ is called a \emph{Fermat prime}\footnote{OEIS sequence A019434.} if $p=2^k+1$. Necessarily, $k=2^\ell$, that is, $p=2^{2^\ell}+1$. There are only five known Fermat primes, namely $3$, $5$, $17$, $257$, and $65537$. It is conjectured that these are the only Fermat primes. We say a prime $p$ is a \emph{safe prime}\footnote{OEIS sequence A005385. Note that the associated prime $q$ to the safe prime $p=2q+1$ is called a \emph{Sophie Germain prime}. Safe primes received their name because of their historical usage in cryptography.} if $p=2q+1$ where $q$ is itself a prime number. The first few safe primes are $5$, $7$, $11$, $23$, $47$, $59$, $83$, $107$, $167$. It is conjectured that there are infinitely many safe primes. 

 We now present the main result.
 \begin{Thm}\label{thm:main} Let $p$ be a Fermat or safe prime. Then all Schur rings over the virtually cyclic group $\Z\times \Z_p$ are of one of the following forms:
 \begin{enumerate}[(i)]
 \item $F[\Z_p]^\H\wedge F[\Z]$ or $F[\Z_p]^\H\wedge F[\Z]^\pm$,
 \item $F[H\times \Z_p]^\H \wedge F[\Z]$ or $F[H\times \Z_p]^\H \wedge F[\Z]^\pm$,
 \item $F[\Z\times\Z_p]^\H$,
 \end{enumerate}
 where $H\le \Z$ and $\H\le \Aut(\Z\times \Z_p)\cong GA(1,p)\times \Z_2$, where $GA(1,p)$ denotes the general affine group over a finite vector space of order $p$. Hence, $\Z\times \Z_p$ is traditional and has only countably many Schur rings.
 \end{Thm}

 The proof of Theorem \ref{thm:main} will be found in Section \ref{sec:proof}. Section \ref{sec:Prelim} contains important properties of Schur rings over infinite groups relevant to this proof. Section \ref{sec:cyclic} considers general structure theorems about primitive sets of Schur rings over $\Z\times \Z_n$ and $\Z\times \Z_p$, where $n$ is any positive integer and $p$ is a prime. Section \ref{sec:ds} concludes with remarks relating Schur rings over $\Z\times\Z_n$ with partitions of difference sets.\\

 \noindent\textbf{Acknowledgements}: All calculations made in preparation of this paper were made using MAGMA \cite{Magma}. The authors are grateful to Stephen Humphries for useful conversations and to the anonymous referees for their most helpful comments.

 \section{Properties of Schur Rings}\label{sec:Prelim}
 In this section, we gather important terminology and properties about Schur rings in general that will be useful for forthcoming proofs. Many of these were known to Wielandt \cite{Wielandt64} in his early work, although they have been extended to the infinite case. Most of these properties will be mentioned without proof.\footnote{These omitted proofs can be found in \cite{InfiniteI}.}

 The partition associated to a Schur ring $\S$ will be denoted $\D(\S)$. The elements of $\D(\S)$ are called the \emph{$\S$-classes} (or \emph{primitive sets} of $\S$). Essentially, an element of the group algebra belongs to a Schur ring $\S$ only if it has constant coefficients across each $\S$-class. We say that a subset $C\subset G$ is an \emph{$\S$-subset} if $C$ is a union of the $\S$-classes. When $G$ is finite, this is equivalent to $\overline{C}\in\S$. We say a subset $H\subseteq G$ is an \emph{$\S$-subgroup} if $H$ is an $\S$-subset and a subgroup of $G$. For any $\S$-subgroup $H$, let $\S_H := \S\cap \FH$, called a \emph{Schur subring}. For any Schur ring, the set of $\S$-subsets forms a lattice closed under intersections and unions and the set of Schur subrings forms a lattice closed under intersections and joins. The associated partitions to these Schur subrings are the common coarsenings and the common refinements of $\D(\S_H)$ and $\D(\S_K)$, denoted $\S_{H\cap K}$ and $\S_{\langle H,K\rangle}$, respectively.

 For an element $\alpha\in \FG$, let  $\supp(\alpha) := \{g \mid \alpha_g\neq{0}\} \subseteq G$, which is called the \emph{support} of $\alpha$ and must necessarily be finite.  
 If $\S$ is a Schur ring over group $G$, $\alpha\in \S$, and $G_\alpha= \{g\in G \mid \alpha g= \alpha\}$ is the stabilizer subgroup, then $G_\alpha$ is an $\S$-subgroup of $G$. When $\alpha=\overline{C}$ is a simple quantity, we will denote $G_\alpha$ as $G_C$. Also, if $H=\langle \supp(\alpha) \rangle$, then $H$ is an $\S$-subgroup. 

 Define the \emph{$n$th Frobenius map} for any integer $n$ by the rule $\alpha^{(n)} := \sum _{g\in{G}} \alpha_g g^n$ whenever $\alpha= \sum _{g\in G} \alpha_g g$. The Frobenius map is a linear map $\FG\to\FG$ such that $\alpha^{(mn)} = \left(a^{(m)}\right)^{(n)}$ for any integers $n$ and $m$ and $\alpha^{(-1)} = \alpha^*$. We define the Frobenius map on subsets of $G$ analogously. For example, all the subgroups of $\Z$ can be written as $\Z^{(n)}$ for some integer $n$. The Frobenius map is very useful in determining the structure of primitive sets of Schur rings over abelian groups. In fact, if $\S$ is a Schur ring over an abelian group $G$ and $m$ is an integer coprime to the orders of all torsion elements of $G$, then for all $\alpha\in \S$ we have that $\alpha^{(m)}\in \S$. Using the Frobenius map, we see that for a Schur ring $\S$ over an abelian group $G$, if the torsion subgroup $T(G)$ has finite exponent then $T(G)$ is an $\S$-subgroup. Hence, Schur rings over abelian groups have a torsion Schur subring, denoted $T(\S)$.

 Let $\varphi : G\to H$ be a group homomorphism. Then this map linearly lifts to the group ring naturally and will be denoted by the same symbol $\varphi :\FG\to \FH$. A ring homomorphism between group rings of this form is called a \emph{Cayley homomorphism} (the Frobenius map is an example of such). 
 If $\S$ is a Schur ring over $G$ and $\ker \varphi$ is an $\S$-subgroup, then $\varphi(\S)$ is a Schur ring over $\varphi(G)$. In particular, $\D(\varphi (\S)) = \{\varphi(C) \mid C\in \D(\S)\}$.  Additionally, if $C\in\D(\S)$, then the non-zero intersection numbers with $C$ and the cosets of $\ker \varphi$ are constant, that is, $|C\cap g(\ker \varphi)|=|C\cap g'(\ker \varphi)|$ whenever these intersections are non-empty. Furthermore, if two $\S$-classes both intersect some coset of $K$, then they intersect all the same cosets of $K$. These facts imply that $\S/T(\S)$ is a Schur ring over $G/T(G)$ for any abelian group $G$. 

 
 Frobenius maps, torsion subgroups, and Cayley projections will prove to be helpful in determining the structure of the primitive sets of Schur rings over $\Z\times \Z_n$. We introduce two other important counting arguments that were not included in \cite{InfiniteI} but which will also prove useful here. 

 The first counting argument is a technique of Scott \cite{Scott} about the sizes of primitive sets. In the next two lemmas, we extend Scott's method to infinite groups. 

 \begin{Lem}[\cite{Scott} 13.8.2] \label{lem:sizes}
 Let $C$, $D$, $E \in \D(\S)$ be primitive sets of a Schur ring $\S$. Suppose that $\overline{E}$ appears in the product $\overline{C}\, \overline{D}$ with nonzero multiplicity $\lambda$. Likewise, let $\mu$ and $\nu$ be the multiplicities of $\overline{C}^*$ and $\overline{D}^*$ in the products $\overline{D}\, \overline{E}^*$ and $\overline{E}^*\overline{C}$, respectively. Then $\lambda|E|=\mu|C|=\nu|D|$.
 \end{Lem}

 In other words, for any $\S$-classes $C$, $D$, $E$, 
 \[\lambda_{CDE}|E| = \lambda_{DE^*C^*}|C|=\lambda_{E^*CD^*}|D|.\]

 \begin{Lem}[\cite{Scott} 13.8.3] \label{lem:primproduct}
 Let $\S$ be a Schur ring. Let $C, D\in \D(\S)$ be primitive sets such that $\gcd(|C|, |D|)=1$. Then $\overline{C}\,\overline{D}=\lambda\overline{E}$ for some primitive set $E\in \D(\S)$.
 \end{Lem}

 As the proofs are the same as found in \cite{Scott} without modification, they are omitted. It should be noted that some assumptions made by Scott are omitted in these lemmas, particularly the assumption that the Schur ring be \emph{primitive},\footnote{We say that a Schur ring $\S$ is \emph{primitive} if it has no nontrivial, proper $\S$-subgroups (this definition, while equivalent, is not the definition originally used by Wielandt or Scott). The primitive Schur rings have been studied extensively, particularly because of their applications to permutation groups. In \cite{Scott}, Scott uses 13.8.2 and 13.8.3 to rule out the existence of primitive Schur rings over groups of small order in a way analogous to how Sylow theorems are used to show the non-existence of simple groups of small order.} as they are not actually used in the proof of 13.8.2 and the portion of 13.8.3 which requires them is not included in Lemma \ref{lem:primproduct} above. 

 For the second counting argument, let $A,B$ be finite subsets of a group $G$ such that $|A|=|B|$. We say that an element $x\in AB^*$ is a \emph{tycoon}\footnote{A \emph{difference set} $D$ is a subset of a group $G$ such that every non-identity element in $G$ appears in $DD^*$ with the same multiplicity. When $G$ is finite, this is equivalent to $\overline{D}\,\overline{D}^* = n+\lambda\overline{G}$. Thus, difference sets guarantee the greatest equity among the multiplicities of group elements in the product $DD^*$. As the multiplicity of any element in the product $AB^*$ ranges between $0$ and $|A|$, a set containing a tycoon represents the greatest possible inequity between multiplicities of group elements. If multiplicities are replaced by wealth for the sake of analogy, the curious name tycoon is then explained.} in $(A,B)$ if the multiplicity of $x$ in $AB^*$ is $|A|$. Note that if $x$ is a tycoon in $(A,B)$, then $x^{-1}$ is a tycoon in $(B,A)$.  

 The simplest example of a tycoon is the group identity in the pair $(A,A)$. In fact, every tycoon $x$ in $(A,B)$ is essentially just this example up to translation. More specifically, $x$ is a tycoon in $(A,B)$ if and only if $A=xB$ if and only if $AB^*=xBB^*=AA^*x$. This can be seen by counting solutions to the equation $ab^{-1}=x$ for $a\in A$, $b\in B$ and a fixed tycoon $x\in AB^*$.

 While the existence of tycoons is fairly trivial, the existence of two tycoons in an $\S$-class will be useful. After all, if a pair $(A,B)$ has two tycoons, say $x$, $y$, then $A=xB=yB$. Thus, $B$ is stabilized by $x^{-1}y$, which implies that $G_B\neq 1$. When $A$ and $B$ are $\S$-subsets, multiple tycoons will provide nontrivial $\S$-subgroups. We summarize this in the following lemma.

 \begin{Lem} \label{cor:left-stab} Let $\S$ be a Schur ring with finite $\S$-subsets $A$, $B$ such that $|A|=|B|$. If the pair $(A,B)$ has two distinct tycoons, then $A$ and $B$ are both unions of cosets of some (necessarily finite) nontrivial $\S$-subgroup.
 \end{Lem}

 We briefly remind the reader of the four types of traditional Schur rings. For a finite group $G$, the partition $\D(\S) = \{1, G\smallsetminus 1\}$ always affords a Schur ring, called the \emph{trivial Schur ring}. At the other extreme, for any group $G$ (finite or infinite) the partition of singletons affords a Schur ring, known as the \emph{discrete Schur ring}. This Schur ring coincides with the group algebra itself.

 If $\H \le \Aut(G)$ is a finite automorphism subgroup, then the set of elements of $\FG$ fixed by $\H$  is a Schur ring over $G$, denoted $\FG^{\H}$ and called the \emph{automorphic Schur ring} (or \emph{orbit Schur ring}) associated to $\H$. The discrete Schur ring is automorphic where $\H=1$. When $G$ is abelian, the automorphic Schur ring associated to $\H = \langle *\rangle$, consisting of inverse pairs, is called the \emph{symmetric Schur ring} and is denoted $\FG^{\pm}$. 

 Suppose $G=H\times K$. If $\S$ and $\T$ are Schur rings over $H$ and $K$, respectively, then the ring $\S\otimes_F\T$ has a Schur ring structure, called the \emph{direct product} of $\S$ and $\T$ and is denoted $\S\times\T$. The associated partition is $\D(\S\times\T) = \{CD\mid C\in \D(\S), D\in \D(\T)\}$. Also, $\FH^\H\times \FK^{\mathcal{K}} = F[H\times K]^{\H\times \mathcal{K}}$, that is, the direct product of automorphic Schur rings is automorphic.

 The final family of traditional Schur rings is due to Leung and Man \cite{LeungI}. Let $H,K\le G$ be two nontrivial, proper subgroups such that $K$ is finite, $K\le H$, and $K\trianglelefteq G$. Let $\S$ be a Schur ring over $H$ with $K$ as an $\S$-subgroup. Suppose $\varphi : G\to G/K$ is the natural quotient map. Then $\varphi(\S)$ is a Schur ring over $G/K$. Let $\T$ be a Schur ring over $G/K$ such that $\T_{H/K} = \varphi(\S)$. Then define the \emph{wedge product} $\S\wedge \T$ by the partition \[\D(\S\wedge \T) = \D(\S) \cup \{\varphi^{-1}(D)\mid D\in \D(\T)\smallsetminus\D(\T_{H/K})\}.\] Under these conditions, $\S\wedge \T$ is a Schur ring over $G$, called the \emph{wedge product} of $\S$ and $\T$. Alternatively, a Schur ring $\S$ over $G$ is a wedge product if there exist nontrivial, proper $\S$-subgroups $H,K\le G$ such $K\le H$, $K\trianglelefteq G$, and every $\S$-class outside $H$ is a union of $K$-cosets (this necessarily implies that $K$ must be finite). In this case, we say that $1<K\le H<G$ is a \emph{wedge-decomposition} of $\S$.

 \section{The Structure of Primitive Sets of Schur Rings over $\Z\times \Z_n$}\label{sec:cyclic}
 We now consider the virtually cyclic group $G=\Z\times\Z_n = \langle z\rangle \times \langle a\rangle$. Let $\S$ be a Schur ring over $G$. As mentioned above, $T(G) = \Z_n$ is an $\S$-subgroup. Hence, $T(\S)$ is some Schur subring over $\Z_n$, to which Leung's and Man's classification theorem applies. Let $\varphi : G\to \Z$ be the natural projection map. Again, $\ker \varphi = \Z_n$ is an $\S$-subgroup. Hence, $\varphi(\S)$ is a Schur ring over $\Z$. As there are only two such Schur rings, the discrete and symmetric rings, $\varphi(\S)$ is equal to one of these. Many of the following proofs will be divided into two cases based upon the image $\varphi(\S)$. For example, if $H\le \Z$ is an $\S$-subgroup, then the Schur subring $\S_H$ maps isomorphically onto $\varphi(\S_H)\le \FZ$. If $H\neq 1$, then $H\cong \Z$, and, hence, $\varphi(\S_H)$ itself is either discrete or symmetric. The structure of $\varphi(\S_H)$ must agree with $\varphi(\S)$, that is, they are either both discrete or both symmetric. As $\S_H\cong \varphi(\S_H)$, $\S_H$ is discrete when $\varphi(\S)$ is discrete and symmetric when $\varphi(\S)$ is symmetric for all nontrivial $\S$-subgroups $H\le \Z$. 

 For another example of these two cases, let $t$ be any integer. If $\varphi(\S)$ is discrete, then $\{z^t\}$ is an $\varphi(\S)$-class. Then $\varphi^{-1}(\{z^t\}) = z^t\Z_n$ is an $\S$-subset (but not necessarily primitive). Likewise, if $\varphi(\S)$ is symmetric, then $\{z^t, z^{-t}\}$ is a $\varphi(\S)$-class. Then $\varphi^{-1}(\{z^t, z^{-t}\})= z^t\Z_n\cup z^{-t}\Z_n$ is an $\S$-subset. In either case, $\{z^t, z^{-t}\}\Z_n$ is an $\S$-subset for all integers $t$, which generates its supporting $\S$-subgroup $\langle z^t\rangle \times \Z_n = \Z^{(t)}\times \Z_n\cong \Z\times \Z_n$.  Our main task is to determine which refinements of these $\S$-subsets are possible based upon assumptions on the $\S$-subgroups. The discrete case is always inherently easier to consider. For example, if we have the discrete case and $\S$ contains a singleton other than the identity $\{1\}$, say $\{a^kz^t\}$, then by Lemma \ref{lem:primproduct}, $(a^{k}z^{t})^sC$ is a primitive set of $\S$ for all $C\in \D(T(\S))$ and all integers $s$. Hence, we have $\S_{\Z^{(t)}\times \Z_n} = F[\langle a^kz^t\rangle]\times T(\S) \cong \FZ\times T(\S)$ under this assumption. Therefore, if $z^t\Z_n$ contains a primitive subset which is a singleton, then $\S_{\Z^{(t)}\times \Z_n}$ is traditional. When $T(\S)$ is automorphic (which is always the case when $n$ is prime), we may also conclude that $\S_{\Z^{(t)}\times \Z_n}$ is automorphic. We will use this fact about singletons frequently throughout the sequel.

 \begin{Thm} \label{lem:primfresh}
 Let $G=\Z\times \Z_n$ for some positive integer $n$. Let $\S$ be a Schur ring over $G$. Suppose $\Z^{(n)}$ is an $\S$-subgroup, and let $C$ be an $S$-subset. Then $C\in \D(\S)$ if and only if $C^{(k)}\in \D(\S)$ for all $k$ such that $\gcd(k,n)=1$.
 \end{Thm}

 \begin{proof}
 First, suppose that $C^{(k)}\in \D(\S)$. As $\gcd(k,n)=1$, we have that $|C|=|C^{(k)}|$. If $C$ is imprimitive, then there exists some proper primitive subset $D\subsetneq C$. Then $D^{(k)}\subsetneq C^{(k)}$, which is a proper $\S$-subset of a primitive set, which is a contradiction. Thus, $C\in\D(\S)$.
 
 Next, suppose that $C\in \D(\S)$. Assume first the discrete case, that is,  $\varphi(\S)=\FZ$. Then there is an integer $t$ such that \[C=\{a^{n_1}z^t,a^{n_2}z^t,\ldots, a^{n_l}z^t\} \subseteq z^t\Z_n.\] Consider \[C^{(k)}=\{a^{kn_1}z^{kt},a^{kn_2}z^{kt},\ldots, a^{kn_l}z^{kt}\} \subseteq z^{tk}\Z_n,\] which is an $\S$-subset. Suppose that $C^{(k)}$ is not primitive. Then a subset of it must be, say $D$. Up to relabeling, we may assume $D=\{a^{kn_1}z^{kt},a^{kn_2}z^{kt},\ldots a^{kn_h}z^{kt}\}$. As $\Z^{(n)}$ is an $\S$-subgroup, we know $\{z^n\}$ is primitive. Also, there exist integers $r, s$ such that $kr+ns=1$. Now consider the $\S$-subset \begin{multline*}z^{nst}D^{(r)} =\{a^{kn_1r}z^{ktr+nst},a^{kn_2r}z^{ktr+nst},\ldots, a^{kn_hr}z^{ktr+nst}\}\\ = \{a^{n_1}z^t,a^{n_2}z^t,\ldots a^{n_h}z^t\}.\end{multline*} This is a strict $\S$-subset of $C$, contradicting it being primitive. Hence, $C^{(k)}$ must be primitive. 

 The case where $\varphi(\S)=\FZ^\pm$ is handled similarly, where the $\S$-subset $z^{nst}D^{(r)}$ is replaced with $(\{z^{nst}, z^{-nst}\}D^{(r)})\cap (\{z^t, z^{-t}\}\Z_n)$.  
 \end{proof}

 \begin{Thm} 
 Suppose $G=\Z\times \Z_n$ for some positive integer $n$. Suppose $\S$ is a Schur ring over $G$. Let $H$ be the maximal such $\S$-subgroup such that $H\le \Z$. Assume $n\mid [\Z:H]$. Then let $K$ be the unique subgroup\footnote{In the case that $[\Z:H]=\infty$, that is, $H=1$, set $K:=1$ and the above result and proof would remain valid.} of $\Z$ such that $[K:H]=n$. Then every primitive set of $\S$ contained in $G\smallsetminus (K\times \Z_n)$ is a union of cosets of some nontrivial $\S$-subgroup of $\Z_n$.
 \end{Thm} 

 \begin{proof}
 First, suppose $\varphi(\S)=F[\Z].$ For a fixed integer $t$, let $C\subseteq z^t\Z_n$ be an $\S$-class. Then there exists some subset $A\subseteq \Z_n$ such that $C=z^tA$. Then consider the product $\mathcal{X}:=\overline{C}\, \overline{C}^{(n-1)} =z^{nt}\overline{A}\, \overline{A}^* \in \S.$  Clearly, the identity $1\in AA^*$ is a tycoon. If this pair has no other tycoons, then $z^{nt}$ is the only element in $\mathcal{X}$ with $|A|$-multiplicity, which would imply that $\{z^{nt}\}$ is an $\S$-subset and $\langle z^{nt}\rangle$ is an $\S$-subgroup. If $t$ is chosen such that $z^t\in \Z\smallsetminus K$, then $z^{nt}\in \Z\smallsetminus H$ and we have contradicted the maximality of $H$ (since $\langle H, z^{nt}\rangle$ would be an $\S$-subgroup). Therefore, $(A,A)$ must have a second tycoon, and so Lemma \ref{cor:left-stab} finishes this case.

 The case that $\varphi(\S)=\FZ^\pm$ is handled similarly. Let $C\subseteq \{z^t,z^{-t}\}\Z_n$ be a primitive set, and let $C=z^tA\cup z^{-t}B$ for $A, B\subseteq \Z_n$. Necessarily, $|A|=|B|$ (otherwise, $\varphi(\S)$ would be discrete). Consider the product \[\mathcal{X}:=(\overline{C}\,\overline{C}^{(n-1)})\circ \overline{\{z^{nt},z^{-nt}\}\Z_n} = z^{nt}\overline{A}\, \overline{A}^*+ z^{-nt}\overline{B}\, \overline{B}^* \in \S.\]
  Clearly, $1$ is a tycoon in $(A,A)$ and $(B,B)$. Then we have that $z^{nt}$ and $z^{-nt}$ have the same multiplicity of $|A|$ in $\mathcal{X}$. If these are the only tycoons in the pairs $(A,A)$ and $(B,B)$, then $\{z^{nt},z^{-nt}\}$ is primitive in $\S$. If again $t$ is chosen such that $z^t\in \Z\smallsetminus K$, then the maximality of $H$ is contradicted. We, therefore, conclude that $(A,A)$ or $(B,B)$ must have a second tycoon. In fact, they both do. To see this, the subset of $\supp(\mathcal{X})$ consisting of all the same fixed multiplicity is clearly an $\S$-subset. In particular, the subset of $\supp(\mathcal{X})$ with elements having multiplicity $|A|$ is an $\S$-subset, call it $D$, but $D$ is exactly the set of the tycoons of $(A,A)$ and $(B,B)$. As $\varphi(D)$ is a symmetric set in $\varphi(\S)$, the number of elements in $D$ of the form $a^iz^{nt}$ must equal the number of elements in $D$ of the form $a^jz^{-nt}$. In particular, the number of tycoons in $(A,A)$ is equal to the number in $(B,B)$. By Lemma \ref{cor:left-stab}, $A$ and $B$ are unions of cosets of some $\S$-subgroup of $\Z_n$. Let $\Z_d$ be the $\S$-subgroup of $\Z_n$ which stabilizes $A$. Then 
  \[\overline{C}\, \overline{\Z_d} = (z^t\overline{A}+ z^{-t}\overline{B})\overline{\Z_d}=|d|z^t\overline{A}+ z^{-t}\overline{B}\, \overline{\Z_d} = |d|\overline{C}, \] where the last equality follows from the primitivity of $C$. Hence, $B\Z_d= B$, and $C$ is a union of cosets of $\Z_d$. 
 \end{proof}

 Let $p$ be a prime number. In the case $G=\Z\times \Z_q$, where $q=p^n$, there is a unique minimal torsion $\S$-subgroup in $\Z_q$. Thus, every $\S$-class outside $K\times \Z_q$ is a union of cosets of this unique minimal $\S$-subgroup. Hence, $\S$ is necessarily a wedge product. Furthermore, if $n$ in the above proof is replaced with a power of a prime $q=p^n$, then we gain greater control on the structure of these primitive sets in $\S$, which leads naturally to the following corollary.

 \begin{Cor}\label{lem:trivial} Let $G=\Z\times \Z_{q}$, where $q=p^n$ for some prime $p$. Let $\S$ be a Schur ring over $G$. Let $H$ be the maximal $\S$-subgroup such that $H\le \Z$. Assume $p\mid [\Z:H]$. Then let $K$ be the unique subgroup of $\Z$ such that $[K:H]=p$.  Then $\S$ is a wedge product. In particular, $\S=\S_{K\times \Z_q}\wedge \varphi(\S)$.
 \end{Cor}

 We next consider the case that $G=\Z\times \Z_p$, for some prime $p$. Note that the classification of Schur rings over cyclic groups simplifies when the cyclic group has prime order. As direct and wedge products are impossible over $\Z_p$ and the trivial Schur ring is automorphic, all Schur rings over $\Z_p$ are automorphic and correspond to an integer $m$ coprime to the order $p$. 

 Before continuing, we stop to discuss $\Aut(\Z\times \Z_p)$. Let $\sigma_m: G\to G$ be the automorphism associated with the rule $\sigma_m(z)=z$, and $\sigma_m(a)= a^{m}$ for each integer $m$ coprime to $p$. Also, let $\rho: G \to G$ be the automorphism defined by $\rho(z)=az$ and $\rho(a)=a$. Finally, let $r$ be a primitive root modulo $p$. Then $\Aut(\Z\times \Z_p) = \langle \rho, \sigma_r, *\rangle \cong GA(1,p)\times \Z_2$, where $GA(1,p)$ denotes the general affine group over a finite vector space of order $p$.

 \begin{Thm} \label{lem:allz-srings} 
 Suppose $G=\Z\times \Z_p$, where $p$ is a prime. If $\S$ is a Schur ring over $G$ where $\Z$ is an $\S$-subgroup, then $\S$ is an automorphic Schur ring.
 \end{Thm}
 \begin{proof}
 First, suppose $\varphi(\S)=\FZ$, which implies that $\S_\Z=\FZ$, that is, $\{z^k\}$ is a primitive set for each integer $k$. If $C\in \D(T(\S))$, then $z^kC$ is likewise primitive in $\S$ by Lemma \ref{lem:primproduct}. Hence, $\S=\FG\times T(\S)$.\footnote{Note that this argument does not require $p$ to be prime and hence applies to all positive integers $n$.} As $\S_\Z$ and $\S_p$ are automorphic, $\S$ is automorphic as well.

 Next, suppose $\varphi(\S)$ is symmetric, which implies that $\S_\Z=\FZ^\pm$, that is, $\{z^k, z^{-k}\}$ is a primitive set for each integer $k$. If $C\in \D(T(\S))$, then $\{z^k, z^{-k}\}C$ is an $\S$-subset for every $k$. If all $\S$-subsets of this form are primitive for every $C$ and every $k$, then again $\S=\FG^\pm\times T(\S)$ is automorphic.
 
 Suppose then that there is some integer $k$ and some primitive set $C\in\D(T(\S))$ such that $\{z^k,z^{-k}\}C$ is imprimitive and that $D \subsetneq \{z^k,z^{-k}\}C$ is a primitive subset. Since $\varphi(D)$ is symmetric, $D= z^kA\cup z^{-k}B$ for $A,B\subseteq C$ and $|A|=|B|$. Since all non-identity elements of $\Z_p$ are generators, we may assume that $a\in A$.  If $\pi: G\to \Z_p$ is the natural projection onto $\Z_p$, then $\pi(\overline{D}) = \pi(\overline{A}+\overline{B}) = \overline{A}+\overline{B}$ and $\pi(\overline{C})=\overline{C}$. But $\supp(\overline{A}+\overline{B}) =A\cup B \subseteq C\in \D(T(\S))$. Hence, $A\cup B=C$ and $A\cap B=\emptyset$. 
 
 Let $m$ be the integer coprime to $p$ which affords the automorphic Schur ring $T(\S)$ over $\Z_p$, which necessarily has even order. 
 We claim \[D = z^k\{a^{m^{2t}} \mid t\in \Z\}\ \cup\ z^{-k}\{a^{m^{2t+1}} \mid t\in \Z\} = \{az^k, a^{m}z^{-k}, a^{m^2}z^k, a^{m^3}z^{-k}, \ldots\}.\] If the claim were false, then, up to relabeling of generators, we may assume $D = \{az^k, a^{m}z^k, a^{m^2}z^{-k}, \ldots\}$. Next, \[D^{(m)}=z^{km}A^{(m)}\cup z^{-km}B^{(m)}=\{a^mz^{km}, a^{m^2}z^{km}, a^{m^3}z^{-km}, \ldots\}\] and \begin{multline*}
     \{z^{k(m-1)}, z^{k(1-m)}\}D=z^{km}A\cup z^{k(2-m)}A\cup z^{k(m-2)}B\cup z^{-km}B=\\ \{az^{km}, a^mz^{km}, a^{m^2}z^{k(m-2)}, \ldots, az^{k(2-m)}, a^mz^{k(2-m)}, a^{m^2}z^{-km}, \ldots\},
 \end{multline*} which are both $\S$-subsets. By Theorem \ref{lem:primfresh}, $D^{(m)}$ is primitive. As $a^{m}z^{km}\in D^{(m)}\cap \{z^{k(m-1)}, z^{k(1-m)}\}D$, we have that $D^{(m)}\subseteq \{z^{k(m-1)}, z^{k(1-m)}\}D$. Then as $a^{m^2}z^{km}\in D^{(m)}$ we have that $a^{m^2}z^{km}\in \{z^{k(m-1)},z^{k(1-m)}\}D$. But this implies that $a^{m^2}z^k\in D$ and $a^{m^2}\in A\cap B\neq \emptyset$, which is a contradiction. This proves the claim.  Let $D'=z^kB\cup z^{-k}A$. Hence, $D\cup D'=\{z^k,z^{-k}\}C$ and $D, D'\in \D(\S)$. 
 
 If $\ell$ is some other integer, note that $\{z^{\ell-k}, z^{k-\ell}\}D\cap \{z^\ell, z^{-\ell}\}C$ is a non-empty $\S$-subset, as they both contain $az^\ell$, but it is a proper subset of $\{z^\ell, z^{-\ell}\}C$ since it does not contain $az^{-\ell}$. Applying the above argument, $\{z^\ell, z^{-\ell}\}C =  E\cup E'$, where $E, E'\in \D(\S)$ and $|E|=|E'|$. 

 Finally, if $C'\in \D(T(\S))$ is some other non-identity class, then $C'=C^{(s)}$ for some integer $s$ coprime to $p$. Let $\{z,z^{-1}\}C=E\cup E'$ be a primitive decomposition. Then $E^{(s)}=E'^{(s)} = \{z^s, z^{-s}\}C'$ and $E^{(s)}, E'^{(s)}\in \D(\S)$. Hence, $\{z^s, z^{-s}\}C'$ is imprimitive, which implies that $\{z^\ell, z^{-\ell}\}C'$ is imprimitive with an analogous primitive decomposition as illustrated above. In summary, each primitive set of $\S$ has the form \[\{a^iz^k, a^{im}z^{-k}, a^{im^2}z^k, a^{im^3}z^{-k}, \ldots\}\] for some $i$. But these are exactly the orbits of the automorphism $\sigma_{-m}*$. Therefore, $\S=\FG^{\langle \sigma_{-m}*\rangle}$.
 \end{proof}

 \begin{Cor} \label{cor:allz-srings} 
 Suppose $G=\Z\times \Z_p$, where $p$ is a prime. If $\S$ is a Schur ring over $G$ where $\langle a^iz\rangle$ is an $\S$-subgroup for any integer $i$, then $\S$ is an automorphic Schur ring.
 \end{Cor}
 \begin{proof}
 The image $\rho^{-i}(\S)$ is a Schur ring over $G$ such that $\rho^{-i}(\langle a^iz\rangle) = \langle z\rangle = \Z$ is an $\rho^{-i}(\S)$-subgroup. By Theorem \ref{lem:allz-srings}, $\rho^{-i}(\S)$ is automorphic. As the automorphic image of an automorphic Schur ring is likewise automorphic,\footnote{If $\S=\FG^\H$ and $\tau\in \Aut(G)$, then $\tau(\S) = \FG^{\tau\H\tau^{-1}}$.} $\S=\rho^i(\rho^{-i}(\S))$ is likewise automorphic. 
 \end{proof}

 \begin{Cor} \label{lem:mincase}
 Suppose $G=\Z\times \Z_p$ where p is a prime.\footnote{In the proof of Theorem \ref{lem:mincase}, we will assume that $p$ is an odd prime. The case when $p=2$ is fairly simple and is already addressed in \cite{InfiniteI}.} Let $\S$ be a Schur ring over $G$ such that $\varphi(\S)=F[\Z]^\pm$. If $C\in \D(\S)$ such that $|C|=2$ and $C\subseteq z\Z_p\cup z^{-1}\Z_p$, then $\S$ is an automorphic Schur ring.
 \end{Cor}

 \begin{proof}
 Suppose $C=\{a^kz,a^\ell z^{-1}\}$ is primitive for some $k$ and $\ell$. Then consider the product $\overline{C}^2=a^{2k}z^{2}+a^{2\ell}z^{-2}+2a^{k+\ell} \in \S$. If $a^{k+\ell}$ is the identity, then $\ell\equiv -k \pmod p$. In this case $\langle C \rangle =\langle a^kz\rangle \cong \Z$, which is an $\S$-subgroup. Then, by Corollary \ref{cor:allz-srings}, $\S$ is an automorphic Schur ring. If $a^{k+\ell}\neq 1$, then this means that $T(\S)$ is discrete, as the primitive set $\{a^{k+\ell}\}$ generates $\Z_p$. Then $a^{-(k+\ell)/2}C = \{a^{(k-\ell)/2}z, a^{-(k-\ell)/2}z^{-1}\}$, which is primitive by Lemma \ref{lem:primproduct}. Then $\langle a^{-(k+\ell)/2}C \rangle\cong\Z$. Again, we have that $\S$ is an automorphic Schur ring.
 \end{proof}

 \begin{Thm}\label{lem:relprime}
 Suppose $G=\Z\times \Z_p$, where $p$ is a prime. Let $\S$ be a Schur ring over $G$ such that $H$ is the maximal $\S$-subgroup contained in $\Z$. Let $T(\S)=F[\Z_p]^{\H}$ be an automorphic Schur ring with $|\H|=m$. If there is a primitive set $C\in \D(\S)$ such that $C\subseteq \Z\smallsetminus H$ and $\gcd(|C|,m)=1$, then $\S$ is a wedge product or an automorphic Schur ring.
 \end{Thm}

 \begin{proof}
 Because of Corollary \ref{lem:trivial} and Theorem \ref{lem:allz-srings}, it suffices to prove the case where $H=\Z^{(p)}$. First, suppose $\varphi(\S)=\FZ$. Suppose that $\gcd(|C|,m)=1$ for some primitive set $C=z^tA$, for some $A\subseteq \Z_p$. By Theorem \ref{lem:primfresh}, we may suppose $t=1$. If $|A|=1$, then $\S$ is an automorphic Schur ring by previous reasoning. Suppose then that $|A|>1$. Without the loss of generality, we may suppose that $a,a^k \in A$ are distinct elements. Take $D\in \D(T(\S))$ such that $a^{k-1}\in D$. By Lemma \ref{lem:primproduct}, since $|D|=m$, we know that $\overline{C}\,\overline{D}=\lambda\overline{E}$, where $E\in \D(\S)$. Necessarily, $E\subseteq z\Z_p$ and $(az)a^{k-1}=a^kz\in CD=E$. As $a^kz\in C\cap E$ is contained in two primitive sets, we conclude that $C=E$, that is, $CD=C$. So, $D\subseteq G_C$, which implies $G_C=\Z_p$. In particular,  $A=\Z_p$, that is, $C=z\Z_p$. Hence, $\S=\S_H\wedge \FZ$.

 The symmetric case is handled similarly. If $\varphi(\S)=\FZ^\pm$ and $\gcd(|C|,m)=1$ for some primitive set $C=zA\cup z^{-1}B$, for some $A, B\subseteq \Z_p$, then the case where $|A|=1$ is handled by Corollary \ref{lem:mincase} as $|C|=2$ and the case where $|A|>1$ is handled just like the discrete case, that is, we conclude that $A=B=\Z_p$ by considering stabilizers. Thus, $\S$ is either automorphic or wedge decomposable. 
 \end{proof}

 \begin{Cor}\label{cor:relprime} 
 Suppose $G=\Z\times \Z_p$ where $p$ is a prime. Let $\S$ be a Schur ring over $G$ such that $\Z^{(p)}$ is an $\S$-subgroup.  Let $T(\S)=F[\Z_p]^{\H}$ be an automorphic Schur ring with $|\H|=m$. If there is a primitive set $C\in \D(\S)$ such that $C\subseteq \Z\smallsetminus \Z^{(p)}$ and $\gcd(|C|,m)=1$, then $\S$ is an automorphic Schur ring.
 \end{Cor}
 \begin{proof} If $\langle a^iz\rangle$ is an $\S$-subgroup for some $i$, then $\S$ is automorphic by Corollary \ref{cor:allz-srings}. So, we may assume that $\Z^{(p)}$ is a maximal $\S$-subgroup. By Theorem \ref{lem:relprime}, $\S$ is either automorphic or a wedge product of the form $\FZ \wedge \FZ$ or $\FZ^\pm \wedge \FZ^\pm$, with decomposition $1< H\times \Z_p\le G$. In the latter two cases, these wedge products are equal to $\FG^{\langle \rho\rangle}$ and $\FG^{\langle \rho, *\rangle}$, respectively.
 \end{proof}

 \begin{Thm}\label{cor:relprime2} 
 Suppose $G=\Z\times \Z_p$ where $p$ is a prime. Let $\S$ be a Schur ring over $G$ such that $\Z^{(p)}$ is an $\S$-subgroup.  Let $T(\S)=F[\Z_p]^{\H}$ be an automorphic Schur ring with $|\H|=m$. If $m$ is a power of a prime, then $\S$ is an automorphic Schur ring.
 \end{Thm}
 \begin{proof} Let $m=q^k$, where $q$ is prime. As $|z\Z_p|=p$, there is some primitive subset $C$ of $z\Z_p$ such that $q\nmid |C|$. Then Corollary \ref{cor:relprime} applies.
 
 Consider the symmetric case. Note that $|\{z,z^{-1}\}\Z_p|=2p$, so if $q$ is an odd prime, then there is some primitive subset $C$ of $\{z,z^{-1}\}\Z_p$ such that $q\nmid |C|$. Then Corollary \ref{cor:relprime} applies again.

 Suppose $m=2^\ell$ for some $\ell\geq 1$. 
 We may choose the primitive set $C=zA\cup z^{-1}B$ so that $A$ has the smallest possible odd length (such a choice exists as the union of all possible sets $A$ is $\Z_p$ and not all lengths could be even). Thus, $\gcd(|A|,m)=1$. If $|A|=1$, then $\S$ would be automorphic by Corollary \ref{lem:mincase}. So, without the loss of generality, we may assume that $a, a^k\in A$. Let $D\in \D(T(\S))$ such that $a^{k-1}\in D$. Thus, $\overline{C}\, \overline{D} = \lambda \overline{C} + \cdots$. Then $\overline{C}^*\overline{C} = \mu\overline{D}+\cdots$ where $\lambda|C| = \mu|D|$ by Lemma \ref{lem:sizes}. So, $2\lambda|A|=\mu m$, which implies that $m\mid 2\lambda$. Hence, $\lambda = 2^{\ell-1}t$ for some nonzero integer $t$. Now, as $1\le \lambda \le \min(|C|,|D|) = \min(2|A|, 2^\ell)$, we have $t=1$ or $t=2$.

 If $t=2$, then $\overline{C}\,\overline{D} = 2^\ell\overline{C}$ as both sides of the equality count the same number of elements. Since $D$ necessarily stabilizes $A$, we see again that $A=\Z_p$, that is, $C=\{z,z^{-1}\}\Z_p$. By Theorem \ref{lem:primfresh}, the primitive sets belonging to $G\smallsetminus (\Z^{(p)}\times \Z_p)$ all have the form $\{z^k,z^{-k}\}\Z_p$ for some integer $k$. As $\S_{\Z^{(p)}\times \Z_{p}} = F[\Z^{(p)}]\times F[\Z_p]^\H$, we have $\S = \FG^{\langle \rho, \sigma_r, *\rangle}$, where $|\sigma_r|=m$ in $\Aut(\Z_p)$. 
 
 In the above argument, the choice of $t=1$ or $t=2$ depends upon the choice of the primitive set $D$. This primitive set could have been interchanged for any other primitive set $D'$ such that $a^{u-v}\in D'$ and $a^u, a^v\in A$. Let $\D := \{D\in \D(T(\S))\mid xy^{-1}\in D \text{ for some } x,y\in C\}$. For any set $D\in \D$, if $t=2$, then the above argument applies, and we are done. Suppose then that $t=1$ for all $D\in \D$. Then $\overline{C}\,\overline{D} = \frac{m}{2}\overline{C} + \cdots$, where the right-hand side so far only accounts for half of the elements in the product $CD$. On the other hand, 
 \[\overline{C}\, \overline{C}^* = 2|A| + \sum_{D\in\D} \mu_DD,\] where $\mu_D|D| = \frac{m}{2}|C|$ by Lemma \ref{lem:sizes}. Hence, $\mu_D = \dfrac{m(2|A|)}{2m} = |A|$. Counting multiplicities in the above equation, the left-hand side contains $|C|^2 = 4|A|^2$ elements, and the right-hand side contains $2|A| + |\D||A|m$ elements. Hence, $4|A|^2 = 2|A|+|\D||A|m$, or, simply, $2|A| = 1+2^{\ell-1}|\D|$. As the left-hand side is clearly even, the right-hand side must also be even, which implies that $\ell=1$, that is, $m=2$ and $\H=\langle *\rangle$. This observation simplifies the equation involving $\overline{C}\,\overline{D}$ to $\overline{C}(a^{k-1}+a^{1-k}) = \overline{C} + \overline{C_1}$, where $C_1\in \D(\S)$ and $|C_1|=|C|$. This follows from the minimality of the choice of $A$. We note that $a^{2-k}z = az(a^{1-k})\in C\{a^{k-1}, a^{1-k}\}$. If $a^{2-k}\in A$, then $az = a^{k}z(a^{1-k}) = a^{2-k}z(a^{k-1})$ appears with multiplicity 2 in $C\{a^{k-1}, a^{1-k}\}$. This implies $\overline{C}(a^{k-1} + a^{1-k}) = 2\overline{C}$, a contradiction. Thus, $a^{2-k}z\in C_1$. 

 Now, consider $\overline{C_1}(a^{k-1}+a^{1-k})$. As $a^{2-k}z\in C_1$, we know $az\in C_1\{a^{k-1}, a^{1-k}\}$. Thus, $\overline{C_1}(a^{k-1}+a^{k-1}) = \overline{C} + \overline{C_2}$, where $C_2\in \D(\S)$ and $|C|=|C_2|$, by similar reasoning as before. If $C_2=C_1$, then $\{a^{k-1}, a^{1-k}\}$ stabilizes $C\cup C_1$. This implies that $C\cup C_1 = \{z,z^{-1}\}\Z_p$, but this implies that $|A|$ divides $p$. Then $|A|=1$ or $|A|=p$. Both of these cases imply $\S$ is automorphic. 

 Suppose then that $C_2\neq C_1$. As $a^{3-2k}z = a^{2-k}z(a^{1-k})\in C_1\{a^{k-1}, a^{1-k}\}$, $a^{3-2k}z\in C$ or $a^{3-2k}z\in C_2$. In the former case, $a^{2-k}z = az(a^{1-k}) = a^{3-2k}z(a^{k-1})$ appears with multiplicity 2 in $C\{a^{k-1}, a^{1-k}\}$, which again implies that $C=C_1$. Thus, $a^{3-2k}z\in C_2$, and this implies that the process continues. 

 Let $C_0=C$. Eventually, there will exist some integer $n$ where $C_n=C_i$ for some  $i<n$. Then $\left(\sum_{j=0}^n \overline{C}_j\right)(a^{k-1}+a^{1-k}) = \sum_{j=0}^n \mu_j\overline{C}_j$, which shows that $\{a^{k-1}, a^{1-k}\}$ stabilizes $\bigcup_{j=0}^n C_j$. This again implies that $|A|=1$ or $|A|=p$. 
 Thus, $\S$ is automorphic if $m=2^\ell$.
 \end{proof}

 \begin{Thm}\label{thm:twosets} Suppose $G=\Z\times \Z_p$ where $p$ is a prime.\footnote{In the proof of Theorem \ref{thm:twosets}, we will assume that $p$ is an odd prime. The case when $p=2$ is fairly simple and is already addressed in \cite{InfiniteI}.} Let $\S$ be a Schur ring over $G$ such that $\Z^{(p)}$ is an $\S$-subgroup. Suppose $z\Z_p$ or $\{z,z^{-1}\}\Z_p$ is the union of exactly two primitive sets. Then $\S$ is an automorphic Schur ring.
 \end{Thm}
 \begin{proof}
 First, consider the discrete case. If $\{z,z^{-1}\}\Z_p$ is a union of two primitive sets, then they must necessarily be $z\Z_p$ and $z^{-1}\Z_p$. Thus, $\S=\FG^{\langle\rho, \sigma_m\rangle}$.  Consider then the case where $z\Z_p=C\cup D$ for primitive sets $C,D\in \D(\S)$. As $p$ is a prime, $\gcd(|C|,|D|)=1$. Let $D$ be chosen so that  $|D|>|C|$. By Lemma \ref{lem:primproduct}, we know that $\overline{C}\,\overline{D}=\lambda\overline{E}$, for some $E\in \D(\S)$ and $E\subseteq z^2\Z_p$. By Theorem \ref{lem:primfresh}, we know that the primitive sets forming $z^2\Z$ are $C^{(2)}$ and $D^{(2)}$. If $E=C^{(2)}$, then $|E|=|C|$ and $\lambda=|D|$, but this implies $\lambda>|C|$, which is a contradiction as the multiplicity of each term is bounded by $\min(|C|,|D|)$. If $E=D^{(2)}$, then $|E|=|D|$ and $\lambda=|C|$. Say that $C=zA$ for some $A\subseteq \Z_p$. Then  \[\overline{C}\,\overline{C}+\overline{C}\,\overline{D}=\overline{C}\,(z\overline{\Z_p})=z^2(\overline{A}\, \overline{\Z_p})=\lambda(z^2\overline{\Z_p})= \lambda\overline{C^{(2)}} + \lambda\overline{D^{(2)}}.\] Since $\overline{C}\,\overline{D}=\lambda\overline{D^{(2)}}$, we have that $\overline{C}\,\overline{C}=\lambda\overline{C^{(2)}}$, that is,  $\overline{C}^2=\lambda\overline{C^{(2)}}$. Thus, $\overline{A}\, (\overline{A^*})^* = \overline{A}\,\overline{A} = \lambda\overline{A^{(2)}}$. As every element in $A^2$ is a tycoon of $(A,A^*)$ ($\lambda = |A|$), if $\lambda\neq 1$ then $A=\Z_p$, which contradicts $\Z_p$ being a union of two primitive sets. Thus, $|A|=1=|C|$, and $\S$ is automorphic, as $C$ is a singleton.

 Now consider the case where $\varphi(\S)=\FZ^\pm$. As $z\Z_p$ is not an $\S$-subset, we need only consider the case $\{z, z^{-1}\}\Z_p=C\cup D$, where $C,D\in\D(\S)$. Let $C=zA\cup z^{-1}B$ and $D=zA'\cup z^{-1}B'$. Necessarily, $\gcd(|A|,|A'|)=1$ as $p$ is prime. Choose $D$ such that $|A| < |A'|$. Consider the product 
 \[\overline{C}\,\overline{D} =  (z^2\overline{A}\,\overline{A'}+z^{-2}\overline{B}\,\overline{B'})+(\overline{B}\,\overline{A'}+\overline{A}\,\overline{B'})\]
 Note that clearly $(z^2AA'\cup z^{-2}BB')\subseteq \{z^{2}, z^{-2}\}\Z_p$ and $(BA'\cup AB')\subseteq{\Z_p}$. By Theorem \ref{lem:primfresh}, the only primitive sets in $\{z^2, z^{-2}\}\Z_p$ are $C^{(2)}$ and $D^{(2)}$. This means that \[z^2\overline{A}\,\overline{A'}+z^{-2}\overline{B}\,\overline{B'}=\lambda \overline{C^{(2)}}+\mu \overline{D^{(2)}},\] for some integers $\lambda, \mu$.
Hence, 
 \[|A||A'| = \lambda|A^{(2)}|+\mu|A'^{(2)}|=\lambda|A| + \mu|A'|,\] which implies that $|A'| \mid \lambda$. But $\lambda$ must be an integer such that $0\le \lambda\le \min(|A|,|A'|) = |A|$, which implies that $\lambda=0$. Similarly, $\mu=|A|$ since $|A|\mid \mu$ and $0\le \mu\le |A|$ ($\lambda$ and $\mu$ cannot both be $0$). Therefore, 
 $z^2\overline{A}\,\overline{A'}+z^{-2}\overline{B}\,\overline{B'} = \mu\overline{D^{(2)}}$, that is, \[\overline{C}\,\overline{D} = \mu\overline{D}^{(2)} + \mathcal{X}\] for some $\mathcal{X}\in T(\S)$. 
  Next, consider  
  \[\overline{C}\,\overline{C} + \overline{C}\,\overline{D} = \overline{C}(z+z^{-1})\overline{\Z_p} = \mu(z^2+z^{-2})\overline{\Z_p}+2\mu\overline{\Z_p} = \mu(\overline{C^{(2)}}+\overline{D^{(2)}}) +2\mu\overline{\Z_p}.\] Considering those elements contained in $(z^2+z^{-2})\overline{\Z_p}$ coming from $\overline{C}\, \overline{C}$, we have \[\overline{C}\,\overline{C}=\mu\overline{C^{(2)}}+\mathcal{Y},\] for some $\mathcal{Y}\in T(\S)$. On the other hand, \[\overline{C}\,\overline{C}=(z\overline{A}+z^{-1}\overline{B})^2=z^2\overline{A}^2+z^{-2}\overline{B}^2+2\overline{A}\,\overline{B}.\] Comparing terms, we have $z^2\overline{A}^2+z^{-2}\overline{B}^2=\mu\overline{C^{(2)}}$. We also know $\overline{C^{(2)}}=(z\overline{A}+z^{-1}\overline{B})^{(2)}=z^2\overline{A^{(2)}}+z^{-2}\overline{B^{(2)}}$. So this means that $\overline{A}^2=\mu\overline{A^{(2)}}$. As $|A|=\mu$, we see all the elements of $A^2$ are tycoons in $(A,A^*)$. If $|A|>1$, then we again have that $A=\Z_p$, which contradicts $z\Z_p\cup z^{-1}\Z_p$ having two primitive subsets. Thus, $|A|=1$, which implies that $\S$ is automorphic by Corollary \ref{lem:mincase}.
 \end{proof}

 \section{Proof of Theorem \ref{thm:main}}\label{sec:proof}
 We summarize the results we have found thus far and how these provide the proof of Theorem \ref{thm:main}. In the special case that $G=\Z\times \Z_p$ for some prime $p$, we have a good understanding of the primitive sets of any Schur ring $\S$ over $G$, which we summarize here. Let $H$ be the maximal $\S$-subgroup contained in $\Z$. We have seen that if $H$ is trivial, then $\S$ is a wedge product by Corollary \ref{lem:trivial}, which is traditional. When $H$ is nontrivial, the Schur subring $\S_{H\times \Z_p}$ is necessarily automorphic by Theorem \ref{lem:allz-srings}. If there is a subgroup $K$ of $\Z$ such that $[K:H]=p$, then $\S$ is a wedge product. If $\S_{K\times \Z_p}$ is traditional, then so is $\S$. Therefore, we know the structure of all primitive sets except those in $\D(\S_{K\times\Z_p})\smallsetminus \D(\S_{H\times \Z_{p}})$. This is why Theorems \ref{lem:relprime}-\ref{thm:twosets} focused on the case where $H=\Z^{(p)}$.

 By Theorem \ref{lem:primfresh}, it suffices to determine the structure of the primitive subsets of $z\Z_p$ or $\{z,z^{-1}\}\Z_p$ (depending on the image $\varphi(\S)$). If $\S$ contains the primitive set $\{a^kz\}$ or $\{a^kz, a^{-k}z^{-1}\}$ for any integer $k$, then Corollary \ref{cor:allz-srings} applies and $\S$ is automorphic. If $\{z,z^{-1}\}\Z_p$ has any primitive subset of length two, then $\S$ is automorphic by Corollary \ref{lem:mincase}. Finally, since the torsion subgroup $\Z_p$ is order $p$, we know that $T(\S)$ is an automorphic Schur ring with all non-identity primitive sets having equal length of $m$. The proofs of Theorems \ref{lem:allz-srings} and \ref{lem:relprime} utilized this fact. Theorem \ref{lem:safeprime} will show in the case that $p$ is a Fermat or safe prime that $\S_{K\times \Z_p}$ is necessarily automorphic, which imply that $\S$ is traditional in this case. For general $p$, the integer $m$ is necessarily a divisor of $p-1$. In the case of a Fermat or safe prime, $p-1$ has very few divisors. This will complete the proof of Theorem \ref{thm:main}.

 \begin{Thm} \label{lem:safeprime}
 Suppose $G=\Z\times \Z_p$ where $p$ is a Fermat or safe prime. If $\S$ is a Schur ring over $G$ where $\Z^{(p)}$ is an $\S$-subgroup, then $\S$ is an automorphic Schur ring.
 \end{Thm}

 \begin{proof}
 Let $H=\Z^{(p)}$. We know that $\S_{H\times \Z_p}$ is an automorphic Schur ring by Theorem \ref{lem:allz-srings}. Let each non-identity primitive set in $T(\S)$ have length $m\mid (p-1)$. If $p$ is a Fermat prime, then $m=2^\ell$ for some $\ell$. If $p=2q+1$ is a safe prime, then $m=1$, $2$, $q$, or $2q$. Thus, we will show that if $m=1$, $2^\ell$, $q$, or $2q$, then $\S$ is automorphic. By Theorem \ref{cor:relprime2}, we need only consider the case $m=2q$ (hence $p=2q+1$ is a safe prime). Note in this case $T(\S)$ is trivial.

 First consider the case where $\varphi(\S)=\FZ$.   Let $C=zA \subseteq z\Z_p$ be a  primitive set. If $|A|=1$, then $\S$ is automorphic. So, we assume $|A|>1$. Note $CC^*=AA^*\subseteq \Z_p$. Say $|A|=k$. By counting multiplicities in $\overline{A}\, \overline{A}^*$ and considering that $T(\S)$ is trivial, then there must be some positive integer $\lambda$ such that $k(k-1)=\lambda(p-1) = 2\lambda q$. In fact, this implies that $A$ is a difference set of $\Z_p$. Since $q\mid k$ or $q\mid (k-1)$, we know $k = 0$, $1$, $q$, $q+1$, $2q$, or $2q+1$. As $A$ is non-empty and $z\Z_p$ does not contain a primitive subset which is a singleton, we may rule out $k=0, 1, 2q$. If $k=2q+1=p$, then $\S=\FG^{\langle\rho, \sigma_r\rangle}$, where $r$ is a primitive root of $p$.  If $k=q+1$, then $z\Z_p$ contains a primitive subset whose length is either $1$ or $q$. Therefore, the only case that needs further pursuit is $k=q$. If $k=q$, then the other primitive set in $z\Z_p$ has length $q+1$, that is, $z\Z_p = C\cup D$ where $D\in \D(\S)$ and $|D|=q+1$. Thus, Theorem \ref{thm:twosets} shows that $\S$ is automorphic. 

 Using similar counting arguments, Corollary \ref{lem:mincase}, and Theorem \ref{thm:twosets} again, we see that $\S$ is also automorphic if $\varphi(\S)$ is symmetric. 
 \end{proof}

 \section{Connection between Schur rings over $\Z\times \Z_n$ and Difference Sets}\label{sec:ds}

 Consider a primitive subset of $z\Z_v$ contained in a Schur ring over $G=\Z\times \Z_v$ (we are assuming the discrete case, as the symmetric case is similar), say $C=zD$. Say $|D|=k$. Suppose that the Schur subring $T(\S)$ is trivial. Then $\overline{C}\,\overline{C}^*\in T(\S)$, that is, $\overline{C}\,\overline{C}^* = n+\lambda \overline{\Z_v}$ for some integer $\lambda$ and $n=k-\lambda$. This implies that $D$ is a difference set of $\Z_v$, as mentioned in the proof of Theorem \ref{lem:safeprime}. As the primitive subset $C$ was arbitrary, $\Z_v$ must be a union of disjoint difference sets in this case. 

 \begin{Def} We say that a partition $\D$ of a finite group $G$ is a \emph{difference partition} if each block in $\D$ is itself a difference set.
 \end{Def}

 There are many simple examples of difference partitions. For example, $\{\{G\}\}$ and $\{\{g\} \mid g\in G\}$ are difference partitions, as both $G$ and $\{g\}$ are trivial difference sets. Likewise, $\{\{g\}, G\smallsetminus\{g\}\}$ is a difference partition for any $g\in G$. Generalizing this last example, let $D$ be any difference set of $G$, then $\{D, G\smallsetminus D\}$ is a difference partition. Also, $\{D\} \cup \{\{g\}\mid g\in G\smallsetminus D\}$ is another difference partition. Finally, let $D$ be the set of quadratic residues in $\Z_p$ for $p\equiv 3\pmod 4$, the Paley difference set. Then $\{\{1\}, D, D^*\}$ is also a difference partition (this example can be generalized using other Paley-Hadamard difference sets). Note that some of these difference partitions are associated to Schur rings over $\Z_v$. 

 Note that in all the previous examples of difference partitions, except maybe the complementary partition $\{D, G\smallsetminus D\}$, all of these difference partitions involve a trivial difference set, usually a singleton. Although the complementary partition $\{D, G\smallsetminus D\}$ does not necessarily  involve a trivial difference set over $G$, it is still quite trivial as a partition. 

 \begin{Def} We say that a difference partition is \emph{trivial} if either it contains a trivial difference set or contains exactly two blocks. Otherwise, we say that a difference partition is \emph{non-trivial}.
 \end{Def}

 As hinted above, all the examples of difference partitions listed herein are trivial. To provide an example of a non-trivial difference partition is more challenging. First, translates of the same difference set always have non-empty intersection, which makes translates unusable for forming a partition. Automorphic images of difference sets are, of course, difference sets, but often are equal to translates of the original difference sets. This theory of multipliers of a difference set is a well-studied topic. Thus, in order for a group to have a difference partition, almost certainly it will need at least two non-equivalent difference sets, a situation which is quite rare (the existence of a non-trivial difference set of a group is itself a fairly rare phenomenon). In the case of $\Z_p$, this is a requirement as $\Z_p$ cannot be partitioned using blocks of all the same size. The two smallest primes that even have two non-equivalent, non-trivial difference sets are $p=31$ and $p=307$, neither of which have block sizes that could form a non-trivial difference partition. It is natural to even ask if there is a non-trivial difference partition.

 \begin{quest} Given a cyclic group of prime order $\Z_p$, does there exist a non-trivial difference partition? How about over an arbitrary cyclic group? Or an arbitrary abelian group?
 \end{quest}

 The counting arguments used in the proof of Theorem \ref{lem:safeprime} show that there is no non-trivial difference partition over $\Z_p$ if $p$ is a Fermat or a safe prime.

 Of course, the partition of $z\Z_v$ in a Schur ring over $\Z\times \Z_v$ requires more than a difference partition. For example, if $zA, zB\in \D(\S)$, then $(zA)(zB)^* = AB^* \subseteq \Z_v$. If $A\neq B$, then $\overline{A}\, \overline{B}^* = \lambda \overline{\Z_p\smallsetminus 1}$. This implies that $|A||B| = \lambda(v-1)$, a similar formula to the classic formula of difference sets, namely $k(k-1)=\lambda(v-1)$. 

 When $T(\S)$ is not trivial, similar properties of partitions on $z\Z_v$ are required, and these partitions can be viewed as generalizations of difference partitions. This is analogous to the fact that Schur rings over $\Z_p$ in a way generalize difference sets over $\Z_p$. For example, the Schur ring over $\Z_p$ which corresponds to the unique automorphism of order $\dfrac{p-1}{2}$ consists of three primitive sets. When $p\equiv 3\pmod 4$, the non-identity classes are Paley difference sets. When $p\equiv 1 \pmod 4$, the non-identity classes are reversible partial difference sets. 

 Due to the remarks and examples given above, the consideration of non-trivial difference partitions will be necessary for further study of Schur rings over $\Z\times\Z_p$. They will likely also be of broader interest in the theory of Schur rings and algebraic combinatorics itself.

 \bibliographystyle{amsplain}
 \bibliography{Srings}
\end{document}